\documentclass[11pt]{amsart}

\usepackage{amsfonts,epsfig}
\usepackage{latexsym,url}
\usepackage{amssymb}
\usepackage{amsmath}
\usepackage{amsthm}
\usepackage{graphics}
\usepackage[all]{xy}
\usepackage{array, booktabs, ctable}

\usepackage[backref]{hyperref}
\newtheorem{defn}{Definition}[section]

\newtheorem{lemma}[defn]{Lemma}

\newtheorem{theorem}[defn]{Theorem}

\newtheorem{proposition}[defn]{Proposition}

\newtheorem{conj}[defn]{Conjecture}

\theoremstyle{definition}

\newcommand{\Q}{\mathbb Q}
\newcommand{\Z}{\mathbb Z}

\newcommand{\F}{\mathbb F}

\newcommand{\Gal}{\operatorname{Gal}}

\newcommand{\GL}{\operatorname{GL}}

\def\diam#1{\langle#1\rangle}

\begin{document}



\bibliographystyle{plain}
\title[Isogenies over number fields]{Isogenies of non-CM elliptic curves with rational $j$-invariants over number fields}

\author{Filip Najman}
\address{Department of Mathematics, University of Zagreb, Bijeni\v{c}ka cesta 30, 10000 Zagreb, Croatia}
\email{fnajman@math.hr}
\address{Department of Mathematics, Massachusetts Institute of Technology, Cambridge, Massachusetts 02139, USA}
\email{fnajman@mit.edu}

%\keywords{modular curves}


% Authors and running title to go on top of each page 
%\pagestyle{myheadings} \markboth{\'Alvaro Lozano-Robledo}{Rank
%over large fields.}

\begin{abstract}
We unconditionally determine $I_\Q(d)$, the set of possible prime degrees of cyclic $K$-isogneies of elliptic curves with $\Q$-rational $j$-invariants and without complex multiplication over number fields $K$ of degree $\leq d$, for $d\leq 7$, and give an upper bound for $I_\Q(d)$ for $d>7$. Assuming Serre's uniformity conjecture, we determine $I_\Q(d)$ exactly for all positive integers $d$.  	
\end{abstract}

\maketitle

%\part{Use this type of header for very long papers only}

\section{Introduction}

Let $E/K$ be an elliptic curve over a number field. If there exists a $K$-rational cyclic isogeny $\phi:E\rightarrow E'$ of degree $n$, this implies that $\ker \phi$ is a $\Gal(\overline K / K)$-invariant cyclic group of order $n$ and we will say that $E/K$ has an \textit{$n$-isogeny}.

When talking about possible isogeny degrees of elliptic curves over number fields, it makes sense to restrict to only elliptic curves without complex multiplication (CM). This is because an elliptic curve $E$ with complex multiplication by an order $\mathcal O$ of an imaginary quadratic field $L$ will have $p$-isogenies for infinitely many primes $p$ over any number field containing $L$. We will restrict to elliptic curves without CM in the whole paper, without further mention.   

Understanding the possible torsion groups and possible degrees of a cyclic isogeny is one of the basic problems in the study of elliptic curves over number fields. After the possible torsion groups \cite{mazur1} and prime degrees of isogenies \cite{mazur2} of elliptic curves over $\Q$ were determined by Mazur, Kenku \cite{kenku2,kenku3,kenku4, kenku5} soon completed the classification of possible degrees (not just of prime order) of isogenies of elliptic curves over $\Q$.        

From then, there has been much progress in understanding the possible torsion groups of elliptic curves over number fields: primes that can divide the order of the torsion of an elliptic curves over number fields of degree $d$ were determined by Kamienny \cite{kamienny} for $d=2$, Parent \cite{parent0,parent} for $d=3$ and bounds for the size of such primes for general $d$ were determined by Merel \cite{merel}.

Unfortunately, there has been much less progress in understanding possible degrees of isogenies. A full list of primes $p$ such that $p$ divides $n$ for some $n$-isogeny of an elliptic curve over a number field of degree $d>1$ is not known, even when one restricts to elliptic curves defined over a single number field $K\neq \Q$. We should mention that, for a fixed number field $K$, Larson and Vaintrob \cite{vailar} recently proved that such a list of possible degrees is finite, assuming the Generalized Riemann Hypothesis.

In this paper, we give a list of primes $I_\Q(d)$ that divide $n$ for some $n$-isogeny of an elliptic curve with $\Q$-rational $j$-invariant without CM over a number field of degree $\leq d$. This can be considered to be an analogue of a similar result of Lozano-Robledo \cite{lozano1} for the torsion, and in fact we will use similar methods as in that paper. 

We should note that when studying $p$-isogenies one can look at the set of elliptic curves with rational $j$-invariant instead of the set of elliptic curves with coefficients from $\Q$. The latter set has density 0 in the former over any number field $\neq \Q$ and using any sensible ordering. We can study just the $j$-invariants because a $p$-isogeny is a quadratic--twist-invariant property, while having $p$-torsion is not (except when $p=2$). In other words the set of elliptic curves with a $p$-isogeny is a coarse moduli space, while the set of elliptic curves with $p$-torsion (for $p> 3$) is a fine moduli space. 

By the aforementioned result of Mazur \cite{mazur2} we know that
$$I_\Q(1)=\{2,3,5,7,11,13,17,37\}.$$
Note that by definition $I_\Q(1)\subseteq I_\Q(d)$ for all $d\geq 1$.

We prove the following result.
\begin{theorem}
\label{maintm}
$I_\Q(d)=I_\Q(1)$ for all $d\leq 7$.
\end{theorem}

We also give an unconditional upper bound on $I_\Q(d)$ for all positive integers $d$ in Theorem \ref{tm_up_bound}.

In Section 4, we describe $I_\Q(d)$ for all positive integers $d$, under the assumption that Serre's uniformity conjecture (see Conjecture \ref{suc}) is true. 

\section{Preliminaries: Galois representations} 

Studying both the torsion and isogenies of elliptic curves can be viewed as a more general problem of studying their Galois representations. Let $E[n]=\{ P\in E(\overline \Q )| nP=0 \}$ denote the $n$-th division group of $E$ over $\overline \Q$ and let $\Q(E[n])$, the field obtained by adjoining the coordinates of all points in $E[n]$, be the $n$-th division field of $E$. The Galois group $G_\Q=\Gal (\overline \Q /\Q)$ acts on $E[n]$ and gives rise to a homomorphism $$\rho_{E,n}: G_\Q \hookrightarrow \GL_2 (\Z/n\Z)$$
called a \emph{mod $n$ Galois representation}. The composition of the determinant map and $\rho_n$ is the \emph{cyclotomic character} $\chi_n$.  For a number field $K$, $E(K)[n]$ denotes the set of $K$-rational points in $E[n]$.

Let $p$ be a prime and $\epsilon$ a fixed quadratic non-residue of $\F_p$. Following \cite{lozano1}, we define
$$\mathcal C_{ns}=\left\{\begin{pmatrix}
a & \epsilon b \\ b & a
\end{pmatrix}: a,b \in \F_p, (a,b)\not\equiv (0,0) \pmod p \right\}$$
to denote the non-split Cartan subgroup of $\GL_2(\F_p)$. Futhermore, we define
$$M(a,b):=\left\{\begin{pmatrix}
	a & \epsilon b \\ b & a
\end{pmatrix}: a,b \in \F_p, (a,b)\not\equiv (0,0) \pmod p \right\},$$
$$ N(c,d):=\left\{\begin{pmatrix}
c & \epsilon d \\ -d & -c
\end{pmatrix}: c,d \in \F_p, (c,d)\not\equiv (0,0) \pmod p \right\}$$
and 
$$\mathcal C_{ns}^+=\left\{M(a,b), N(c,d), a,b,c,d\in \F_p, (a,b),(c,d) \not \equiv (0,0) \pmod p\right\}$$
be the normalizer of the non-Split Cartan subgroup.

Let $E/ \Q$ be an elliptic curve and $p\geq 5$ a prime, and let $K$ be an extension of $\Q_p$ of the least possible degree such that $E$ has good or multiplicative reduction over $K$. Let $e$ be the ramification index of $K$ over $\Q_p$; it is well known that $e \leq 6$ \cite{serre1}.

\begin{theorem}[\cite{bilu, pbr, mazur2, serre1}]
\label{brp}
Let $p\notin I_\Q(1)$ be a prime and $e$ be the ramification index of $K/ \Q_p$, as defined above. 
Then the image $G$ of $\rho_{E,p}(G_\Q)$ is either 
\begin{enumerate}

\item Contained in the normalizer of a non-split Cartan subgroup: then $G$ contains the $e$-the power of a non-split Cartan subgroup, or

\item Surjective, i.e. $G=\GL_2(\F_p)$. 
\end{enumerate}
\end{theorem}

In fact, Zywina \cite{zyw} recently proved an even more precise result of what the image of $\rho_{E,p}$ looks like if $p\notin I_\Q(1)$.

\begin{proposition}[\cite{zyw}, Proposition 1.13.]
\label{prop-zyw}
Suppose $E/\Q$, $p\notin I_\Q(1)$ and $\rho_{E,p}$ is not surjective. Then
\begin{enumerate}
\item If $p\equiv 1 \pmod 3$, then $\rho_{E,p}(G_\Q)$ is conjugate in $\GL_2(\F_p)$ to $\mathcal C_{ns}^+$.
\item If $p\equiv 2 \pmod 3$, then $\rho_{E,p}(G_\Q)$ is conjugate in $\GL_2(\F_p)$ to either $\mathcal C_{ns}^+$ or $(\mathcal C_{ns}^+)^3.$
\end{enumerate}
\end{proposition}

Serre's Open image theorem \cite{serre1} implies that for an elliptic curve $E/ \Q$ without CM, for all but finitely many primes $p$, $\rho_{E,p}$ is surjective. 

We should note that there does not exist one known elliptic curve $E / \Q$ such that for a prime $p\notin I_\Q(1)$, the representation $\rho_{E,p}$ is not surjective. Sutherland \cite{sut} has recently checked this for all elliptic curves \cite{cre} (over $2$ million of them) with conductor up to $350000$ and all elliptic curves in the Stein-Watkins database (more than 140 million curves).

These observations gives rise to Serre's Uniformity conjecture states that there should exist a bound $B$, not depending on the elliptic curve $E$, such that $\rho_{E,p}$ is surjective for all $p> B$ and for all elliptic curves over $\Q$. Here we state the following version of this conjecture. 

\begin{conj}[Serre's uniformity conjecture, see \cite{zyw}, Conjecture 1.12.]
\label{suc}
For $E/\Q$, $p\notin I_\Q(1)$, the representation $\rho_{E,p}$ is surjective. 
\end{conj}

\section{Degree of the field of definition of a $p$-isogeny}

To prove Theorem \ref{maintm}, we will need to find the minimal degree of definition of a $p$-isogeny of an elliptic curve with $\Q$-rational $j$-invariant. By Theorem \ref{brp} and Proposition \ref{prop-zyw}, we need to consider 2 cases: either $\rho_{E,p}$ is surjective or its image is surjective or is contained in a normalizer of non-split Cartan subgroup.

Let $P\in E[p]$ be a point of degree $p$ and $C=\diam{P}$ be the subgroup generated by $P$. For a number field $K$, we define $K(P)$ to be the field obtained by adjoining the coordinates of $P$ to $K$ and $K(C)$ to be smallest extension of $K$ such that the $p$-isogeny $\phi$ with kernel $C$ is defined over $K$, or in other words, the smallest number field such that $\Gal(\overline{K(C)}/K(C))$ acts on $C$.

\subsection{Full image}
\begin{proposition}
	Let $E/\Q$ be an elliptic curve and $p$ a prime such that $\rho_{E,p}$ is surjective, and $C$ of $E[p]$ of order $p$. Then $[\Q(C):\Q]=p+1$.  
\label{full}
\end{proposition}
\begin{proof}
	Let $\{P,R\}$ be the basis of $E[P]$. The field of definition of $\Q(C)$ is then the fixed field of the subgroup
	$$
	H=\left\{ \begin{pmatrix}
		a & b\\ 0&c
		\end{pmatrix}: a,c \in \F_p^\times, b \in \F_p \right\}.
	$$ 
	We have $|\GL_2(\F_p)|=p(p-1)^2(p+1)$ and $|H|=p(p-1)^2$, so we conclude that 
	$$[\Q(\diam{P}):\Q]=|\GL_2(\F_p)/H|=p+1.$$
\end{proof}

\subsection{Normalizer of non-split Cartan}
\label{sec:nsplit}

A result that we will need is the following easy lemma.
	
\begin{lemma}
\label{ent}
Let $E/K$ be an elliptic curve over a number field and $P\in E[p]$. Let $C=\diam{P}$. Then $[K(P):K(C)]$ divides $p-1$.
\end{lemma}
\begin{proof} By definition $E$ has a $p$-isogeny over $K(C)$. Then the same proof as \cite[Lemma 7]{ent}, taking $K(P)$ instead of $\Q$ as the base field (which does not make a difference in the proof) proves the claim.
\end{proof}

Now we can prove our result. 

\begin{proposition}
	Let $E/\Q$ be an elliptic curve and $p$ a prime such that the image of $\rho_{E,p}$ is contained in the normalizer of the non-split Cartan subgroup and let $\diam{P}=C\subseteq E[p]$ a cyclic subgroup of order $p$. Then 
\begin{enumerate}

\item If $p\equiv 2 \pmod 3$, then  $[\Q(C):\Q]\geq p+1$.

\item If $p\equiv 2 \pmod 3$, then  $[\Q(C):\Q]\geq (p+1)/3$.

\item If $E$ does not have additive reduction at $p$, then  $[\Q(C):\Q]\geq p+1$.

\end{enumerate}	
\label{nsp}
\end{proposition}
\begin{proof}
	For an elliptic curve $E/\Q$ such that the image of $\rho_{E,p}$ is contained in the normalizer of the non-split Cartan subgroup, by the proof of \cite[Theorem 7.3]{lozano1} the field of smallest degree $\Q(P)$ over which a point $P$ of order $p$ is defined is $\geq \frac{p^2-1}{a},$ where $a$ is the smallest integer such that $\rho_{E,p}(G_\Q)$ contains an $a$-th power of $\mathcal C_{ns}$\footnote{In the statment of \cite[Theorem 7.3]{lozano1}, it says that $[\Q(P):\Q]\geq \frac{p-1}{e}$, where $e$ is as defined in Section 2, since this guarantees that $\rho_{E,p}(G_\Q)$ contains $C_{ns}^e$. But from the proof we see that it is true that $[\Q(P):\Q]\geq \frac{p^2-1}{a}$, where $a$ (which may be smaller than $e$) is the smallest integer such that $\rho_{E,p}(G_\Q)$ contains $C_{ns}^a$.}.

	On the other hand, by Lemma \ref{ent}, we have $[\Q(P):\Q(C)]\leq p-1$. Together this implies that for any $P \in E[p]$,
	$$ [\Q(C):\Q]\geq \frac{\frac{p^2-1}{a}}{p-1} \geq \frac{p+1}{a}.$$ 
By Proposition \ref{prop-zyw},
$$a=\begin{cases}1  & \text{ if } p\equiv 1 \pmod 3,\\
1 \text{ or } 3 & \text{ if }p\equiv 2 \pmod 3, \end{cases}$$ from which (1) and (2) follow.

Part (3) follows from part (1) of Theorem \ref{brp}, since by assumption we have $e=1$ and hence $\rho_{E,p}(G_\Q)$ contains $\mathcal C_{ns}$.
\end{proof}

\subsection{Proof of Theorem \ref{maintm}}
Let $p \notin I_\Q(1)$ and $d(p)$ be the minimal field of definition of a $p$-isogeny of an elliptic curve with rational $j$-invariant. 

By Propositions \ref{full} and \ref{nsp}, we have that $d(p)\geq p+1$ if $p\equiv 1 \pmod 3$ and $d(p)\geq (p+1)/3$ if $p\equiv 2 \pmod 3$.
Since for $p\notin I_\Q(1)$, we have $p\geq 19$ for $p\equiv 1 \pmod 3$ and $p\geq 23$ for $p\equiv 2 \pmod 3$, it follows that $d(p)\geq 8$ for all $p\notin I_\Q(1)$, proving the Theorem. \qed

Note that we have in the proof above in fact proved an unconditional upper bound for $I_\Q(d)$, for all integers $d$.

\begin{theorem}
\label{tm_up_bound}
For all positive integers $d$, 
$$I_\Q(d)\subseteq I_\Q(1) \cup \left\{ p \text{ prime}: p\leq d-1,\ p \equiv 1 \pmod 3 \right\} \cup \left\{ p \text{ prime}: p\leq 3d-1,\ p \equiv 2 \pmod 3 \right\}.$$
\end{theorem}

\section{Results assuming Serre's uniformity conjecture}
If we assume Conjecture \ref{suc}, we can prove stronger results.

\begin{theorem}
Suppose Conjecture \ref{suc} is true. Then for all positiver integers $d$,
$$I_\Q(d)=I_\Q(1) \cup \left\{ p \text{ prime}: p\leq d-1\right\}.$$
In particular $I_\Q(d)=I_\Q(1)$ for $d \leq 19$. 
\end{theorem}
\begin{proof} Let $p \notin I_\Q(1)$ and $d(p)$ be the minimal field of definition of a $p$-isogeny of an elliptic curve with rational $j$-invariant. Then by assumption $\rho_{E,p}$ is surjective and by Proposition \ref{full}, $d(p)=p+1$.

\end{proof}

\end{document}